\documentclass[12pt]{amsart}
\usepackage{amsmath,amsthm,amscd,amsfonts,amssymb,enumerate}
\usepackage{graphicx}
\usepackage{color}
\usepackage[colorlinks]{hyperref}
\usepackage{amsfonts,amssymb,amscd,amsmath,enumerate,url,verbatim}
 \usepackage[dvips]{epsfig}
\usepackage{amsmath,amssymb,amsfonts,enumerate,amsthm}
 \usepackage{amsgen, amstext,amsbsy,amsopn, amsthm, amsfonts,amssymb,amscd,amsmat
 h,euscript,enumerate,url,verbatim,calc,xypic}
 \usepackage{latexsym}
 \usepackage{graphics}
 \usepackage{color}
\headsep=1cm \numberwithin{equation}{section}
\numberwithin{equation}{section}

\numberwithin{equation}{section}
\input xy
\xyoption{all}

\newtheorem{thm}{Theorem}[section]

\newtheorem{cor}[thm]{Corollary}
\newtheorem{lem}[thm]{Lemma}
\newtheorem{prop}[thm]{Proposition}
\newtheorem{defn}[thm]{Definition}

\newcommand{\Hom}{\operatorname{Hom}\,}
\newcommand{\Ext}{\operatorname{Ext}\,}

\newcommand{\Ass}{\operatorname{Ass}\,}
\newcommand{\Assh}{\operatorname{Assh}\,}
\newcommand{\Att}{\operatorname{Att}\,}
\newcommand{\Supp}{\operatorname{Supp}\,}

\newcommand{\grad}{\operatorname{grade}\,}

\renewcommand{\dim}{\operatorname{dim}\,}

\newcommand{\cd}{\operatorname{cd}\,}

\newcommand{\Min}{\operatorname{Min}\,}

\newcommand{\h}{\operatorname{ht}\,}

\newcommand{\N}{\mathbb{N}}

\newcommand{\fa}{\mathfrak{a}}
\newcommand{\fb}{\mathfrak{b}}

\newcommand{\fm}{\mathfrak{m}}
\newcommand{\fp}{\mathfrak{p}}
\newcommand{\fq}{\mathfrak{q}}

\newcommand{\fx}{\mathfrak{x}}

\begin{document}
\bibliographystyle{amsplain}


\title[Attached and assoiciated primes Of Local Cohomology modules Via linkage  ]
 {Attached and associated primes Of Local Cohomology modules Via linkage}

\bibliographystyle{amsplain}

     \author[Maryam jahangiri]{Maryam jahangiri$^1$}
     \author[khadijeh sayyari]{khadije sayyari$^2$}

\address{$^{1, 2}$ Faculty of Mathematical Sciences and Computer,
Kharazmi University, Tehran, Iran.}

\email{ jahangiri@khu.ac.ir, std-sayyari@khu.ac.ir}





\begin{abstract}
Let $R$ be a commutative Noetherian ring and $M$ be a finitely generated $R$-module. Considering the new concept of linkage of ideals over a module, we study associated prime ideals, cofiniteness and Artinianness of local cohomology modules of $M$ with respect to some linked ideals over it.

\end{abstract}

\maketitle

\bibliographystyle{amsplain}
\section{introduction}

Let $R$ be a commutative Noetherian ring with $1\neq0$, $\fa$ be an ideal of $R$ and $M$ be an
 $R$-module. For $i\in \mathbb{Z}$, the $i$-th local cohomology functor with respect to $\fa$ is
  defined to be the $i$-th right derive functor of the $\fa$-torsion functor $\Gamma_{\fa}(-),$ where
  $\Gamma_{\fa}(M)= \cup_{n\in \mathbb{N}_0} 0:_M\fa^n. $ Local cohomology was defined by Grothendieck \cite{G},
   actually, it is an "algebraic child of geometric  parents". For more details of local cohomology modules, we
   refer the reader to \cite{BS}.

 There are lots of problems in the study of local cohomology modules (see  \cite{H}) and finiteness problems
 in this subject attracts lots of interests. One of the main problems in this topic is finiteness of the set
 of associated prime ideals, i.e. $\Ass H^i_{\fa} (M)$. Although, in \cite{SI}, Singh showed that $\Ass H^i_{\fa} (M)$
 might be infinite, but there are some cases where it is a finite set, see for example \cite{HS},
  \cite{He} and \cite{LY}. One more problem, is Artinianness of $H^i_{\fa} (M)$ and it has been studied
   by many authors too, see for example \cite{DY} and \cite{He2}.

   Another important topic in commutative algebra and algebraic geometry is the theory of linkage. The significant work of
   Peskine and Szpiro \cite{PS} stated this theory in the modern algebraic language; two proper ideals $\fa$ and $\fb$ in $R$ are
    said to be linked if there is an regular sequence $\underline{\fx}$ in their intersection such that
     $\fa = (\underline{\fx}) :_R \fb$ and $\fb = (\underline{\fx}) :_R \fa$.

In a recent paper \cite{JS}, inspired by the works in the ideal case, the authors present the concept of the linkage
 of ideals over a module. Let   $M$ be a finitely generated $R$-module.
Let $\fa$, $\fb$ and $I$ be ideals of $R$ with  $I\subseteq \fa \cap \fb$  such that $I$ is generated by an $M$-regular sequence and
 $ \fa M \neq M\neq \fb M$. Then, $\fa$ and $\fb$ are said to be linked by $I$ over $M,$ denoted by $\fa\sim_{(I;M)}\fb,$ if
 $\fb M = IM:_M\fa$ and $\fa M = IM:_M\fb $. This is a generalization of the classical concept of linkage when $M= R$.

In this paper, we consider the above generalization and study Artinianness and associated prime ideals of local cohomology modules
 $H^i_{\fa} (M)$ where $\fa$ is a linked ideal over   $M$.

 More precisely, in Section 2, we show that if $R$ is
  Cohen-Macaulay and $t\in \mathbb{N}$ then, for any ideal $\fa$ of $R,$ $\Ass H^i_{\fa} (R)$ is finite if and only if
   $\Ass H^i_{\fa} (R)$ is finite for any linked ideal $\fa$ of $R$ (Theorem \ref{t2}).
Then, we study the finiteness of some  $\Ext$ modules and, as a
corollary, we show that if $\fa\sim_{(I;R)}\fb$ then
 $H^i_{\fa} (R)$ is $\fa$-"cofinite" and $\fb$-"cofinite" if and only if it is $I$-"cofinite" (Corollary \ref{c2}).

In Section 3, we study Artinianness and attached prime ideals of
local cohomology modules $H^i_{\fa} (M)$ where $\fa$ is a linked
ideal over  $M$ and, among other things, we present some necessary
and sufficient conditions for the finitely generated $R$-module $M$
to be Cohen-Macaulay in terms of the existence of some special
linked ideals over it (Theorem \ref{t6}).

Throughout the paper, $R$ denotes a non-trivial commutative Noetherian ring, $\fa$ and $\fb$ are non-zero proper ideals
 of $R$ and $M$ will denote a finitely generated $R$-module.

 \section{Associated prime ideals and cofiniteness}

In this section, we study finiteness of the set of Associated prime ideals of local cohomology modules and the "cofinite"
 property of these modules over some linked ideals.

We begin by the definition of one of our main tool.
\begin{defn}\label{F1}
 \emph{Assume that $\fa M\neq M\neq\fb M$ and let $I\subseteq \fa \cap \fb$ be an ideal generating by an $M$-regular
 sequence. Then we say that the ideals $\fa$ and $\fb$ are linked by $I$ over $M$, denoted by $\fa\sim_{(I;M)}\fb$, if
  $\fb M = IM:_M\fa$ and $\fa M = IM:_M\fb $. Also, the ideals $\fa$ and $\fb$ are said to be geometrically linked by $I$
   over $M$ if $\fa M \cap \fb M = IM$. The ideal $\fa$ is $M$-selflinked by $I$ if $\fa\sim_{(I;M)}\fa$. Note that in the
   case where $M = R$, this concept is the classical concept of linkage of ideals in \cite{PS}.}
\end{defn}

\begin{prop}\label{l1}
 Let $I$ be an ideal of $R$ such that $\fa\sim_{(I;M)}\fb$. Assume that $\Ass \frac{M}{IM}=\Min \Ass \frac{M}{IM}$.
 Then the following statements hold.
\begin{itemize}

 \item [(i)] $\Supp H^i_{\fa} (M) \cap \Ass \frac{M}{\fa M}= \emptyset$, for all $i>\grad_M \fa$.
 \item  [(ii)]  If $I=0$ and $M$ is projective then $\Ass \Ext^1_R(\frac{M}{\fa M},\frac{R}{\fa})=
  \Supp \frac{M}{\fb M} \cap \Ass \frac{R}{\fa}$. In particular, in the case where $M= R$, $$\Ass \Ext^1_R(\frac{R}{\fa},\frac{R}{\fa})=
   \Ass \frac{R}{\fb} \cap \Ass \frac{R}{\fa}= \Ass \Ext^1_R(\frac{R}{\fb},\frac{R}{\fb}).$$
     \end{itemize}
  \end{prop}

\begin{proof}
\begin{itemize}
 \item [(i)] Let $\fp\in\Supp H^i_{\fa} (M) \cap \Ass \frac{M}{\fa M}$ for some $i>\grad_M \fa$. Hence,
  by \cite[3.11(i)]{JS}, $\h_M\fp= \grad_M \fa$ and $H^i_{\fa R_{\fp}} (M_{\fp}) =0,$ which is a contradiction.
 \item  [(ii)]  Applying $\Hom_R(-,\frac{R}{\fa})$ on the sequence $ 0\rightarrow \fa M \rightarrow M\rightarrow
  \frac{M}{\fa M}\rightarrow 0,$ we get the exact sequence $$ 0\rightarrow \Hom_R(\frac{M}{\fa M},\frac{R}{\fa})
   \rightarrow \Hom_R(M,\frac{R}{\fa})\rightarrow \Hom_R(\fa M,\frac{R}{\fa})\rightarrow \Ext^1_R(\frac{M}{\fa M},
   \frac{R}{\fa})\rightarrow 0.$$ This, in conjunction with the isomorphisms $$\Hom_R(\frac{M}{\fa M},\frac{R}{\fa})
    \cong \Hom_R(M,\Hom_R(\frac{R}{\fa},\frac{R}{\fa}))\cong \Hom_R(M,\frac{R}{\fa}),$$ implies that $$\Hom_R(\fa M,\frac{R}{\fa})
    \cong \Ext^1_R(\frac{M}{\fa M},\frac{R}{\fa}).$$ Now, in view of the assumption and \cite[2.5(ii)]{JS},
     $$\Ass \Ext^1_R(\frac{M}{\fa M},\frac{R}{\fa})= \Supp \frac{M}{\fb M} \cap \Ass \frac{R}{\fa}.$$

     Also, by \cite[2.8(iii)]{JS}, $\Ass \frac{R}{\fa}= \Ass R \cap V(\fa).$ This proves the last claim.
     \end{itemize}

\end{proof}
The following theorem provides an equivalent condition for the finiteness of $\Ass H^t_{\fa}(R).$
\begin{thm} \label{t2}
Let $R$ be a Cohen-Macaulay local ring and let $t\in \mathbb{N}$. Then, the following statements are equivalent.
\begin{itemize}
 \item [(i)]  For any ideal $\fa$ of $R$, $\Ass H^t_{\fa}(R)$ is a finite set.
 \item [(ii)] For any linked ideal $\fa$, $\Ass H^t_{\fa}(R)$ is a finite set.
 \end{itemize}
\end{thm}

\begin{proof}
Let $\fa \unlhd R$ be an ideal and assume that, for all linked ideals $\fb$, $\Ass H^t_{\fb}(R)$ is a finite set.
 By \cite[Corollary 1]{He}, we may consider $\h \fa = t-1 $. Using \cite[2.11 ]{JS1}, there exists a linked radical
  ideal $\fa'\supseteq \fa$ such that $\grad_R \fa' = \grad_R \fa = t-1$. In view of the structure of $\fa'$
  (in the proof of \cite[ 2.11(i)]{JS1}, and the Cohen-Macaulayness of $R$, we have $$\Ass \frac{R}{\fa'} =
   \{\fp| \fp \in \Min \Ass \frac{R}{\fa}, \h \fp = \h \fa \}.$$ Set $\fb : = \cap_{\fp \in \Min \Ass \frac{R}{\fa}
    - \Ass \frac{R}{\fa'}} \fp$. Then $\h \fb > t-1$ and $\sqrt{\fa} = \fa' \cap \fb$. We claim that $\h \fa'+\fb >t$.
    For that, if $\h \fa'+\fb =t$ then there exists $\fq \in \Min \Ass \frac{R}{\fa'+\fb}$ with $\h \fq = t$. Hence
     $\fq \in \Min \Ass \frac{R}{\fb} \cap V(\fa')$ and there exists $\fp \in \Ass \frac{R}{\fa'}$ such that $\fp\subseteq \fq$,
     which is a contradiction.

Now, the Mayer-Vietorise sequence $$0 \longrightarrow H^t_ {\fa'} (R) \oplus H^t_ {\fb} (R) \longrightarrow H^t_ {\fa}(R)\longrightarrow H^{t+1}_ {\fa'+\fb} (R),$$ in conjunction with \cite[Theorem 1]{He}, proves the claim.

\end{proof}

\begin{prop} \label {p1}
Let $R$ be a UFD and $\fa$ be a linked ideal. Then, $\Ass
H^2_{\fa}(R)$ is finite. If, in addition, $\dim R <4$ then $\Ass
H^i_{\fa}(R)$ is a finite set for all $i\in \N_0.$
\end{prop}

\begin{proof}
     In the case $\grad \fa \geq 2$, the result follows from \cite[Theorem 1]{He}. Let $\fa$ be a linked ideal by $I$ with $\grad \fa =1$. Via \cite[2.3(iv) and 2.6(i)]{JS2}, $\sqrt{\fa} = \fp_1\cap...\cap\fp_l$ for some $\fp_,...,\fp_l \in \Ass\frac{R}{I}.$

     We claim that $\h \fp_i = 1$, for all $i= 1,...,l$. Let $i\in \{ 1,...,l\}$. There exists an irreducible element $x \in \fp_i - Z(R)$. If $\h \fp_i > 1$ then there exists an irreducible element $y \in \fp_i - (x)$. As $\grad \fp_i =1$, $y \in Z(\frac{R}{(x)})$ and there exist $r \in R-(x)$ and $r'\in R$ such that $ry = r'x$. Therefore, $x|r$ which is a contradiction.

     Therefore, by \cite[Exercise 20.3]{M}, $\sqrt{\fa}$ is principal and $H^2_{\fa}(R)=0$.

The last assertion follows from the fact that $H^{\dim R}_{\fa}(R)$ is Artinian.

\end{proof}
The $R$-module $X$ is said to be $\fa$-cofinite if $\Supp X
\subseteq V(\fa)$ and $\Ext^i_R(\frac{R}{\fa}, X)$ is a finitely
generated $R$-module for all $i\in \N_0.$
The cofinite property of
local cohomology modules is one of the main problems in this
subject, see for example \cite{MD}. In the last item of this
section, we consider this problem for linked ideals.

We recall the following lemma from \cite[Proposition 1]{MD} which will be used in the next theorem.
\begin{lem} \label{l3}
Let $N$ be an $R$-module and $p\geq0$. Suppose that
$\Ext^i_R(M, N)$ is a finitely generated $R$-module for all $i\leq p$.
Then, for any finitely generated $R$-module $L$ with $\Supp L \subseteq \Supp M$, $\Ext^i_R(L, N)$ is
 finitely generated for all $i\leq p.$
\end{lem}

The following theorem considers some equivalent conditions for the
finiteness of some Ext modules.
\begin{thm} \label{t3}
Let $I$ be a non-prime ideal of $R$ which is generated by an $R$-sequence and $N$ be an $R$-module. Then, the following statements are equivalent.
\begin{itemize}
 \item [(i)]  $\Ext^i_R(\frac{R}{I}, N)$ is finitely generated, for all $i\in \N_0$.
 \item [(ii)] $\Ext^i_R(\frac{R}{\fp}, N)$ is finitely generated, for any ideal $\fp \in \Ass \frac{R}{I}$ and all $i\in \N_0$.
 \item [(iii)] $\Ext^i_R(\frac{R}{\fa}, N)$ is finitely generated, for any linked ideal $\fa$ by $I$ and all $i\in \N_0$.
 \item [(iv)] $\Ext^i_R(\frac{R}{\fa}, N)$ and $\Ext^i_R(\frac{R}{\fb}, N)$ are finitely generated, for some ideals
  $\fa$ and $\fb$ such that $\fa\sim_{(I;R)}\fb$ and all $i\in \N_0$.
\end{itemize}
\end{thm}

\begin{proof}

 $"(i) \rightarrow (ii)"$ is clear from the fact that $\Supp \frac{R}{\fp} \subseteq \Supp \frac{R}{I}$ and the above lemma.

 $"(ii) \rightarrow (iii)"$ Let $\fa$ be a linked ideal by $I.$ As $\Supp \frac{R}{\sqrt{\fa}} \subseteq \Supp \frac{R}{\fa}$, in view of \ref{l3}, we can assume that $\fa$ is a radical linked ideal by $I$. Let $\Min \Ass \frac{R}{\fa}=\{ \fp_1,..., \fp_n\}$ and $M:= \frac{R}{\fp_1}\oplus ... \oplus\frac{R}{\fp_n}$. By \cite[Proposition 5.p594]{MS}, $\fp_j \in \Ass \frac{R}{I}$. So, $\Ext^i_R(M, N)$ is finitely generated for all $i\geq 0$. Now, the result follows from the fact that $\Supp \frac{R}{\fa}= \Supp M$.

 $"(iii)\rightarrow (i)"$ Let $\Ass \frac{R}{I}=\{ \fp_1,..., \fp_n\}$ and $M:= \frac{R}{\fp_1}\oplus ... \oplus\frac{R}{\fp_n}$. By the assumption and \cite[2.3]{JS2}, $\Ext^i_R(M, N)$ is finitely generated, for all $i\geq 0$. Therefore, the result has desired from the above lemma.

 $"(iv)\rightarrow (i)"$  Assume that $\Ass \frac{R}{I}=\{ \fp_1,..., \fp_n\}$ and $M:= \frac{R}{\fp_1}\oplus ... \oplus\frac{R}{\fp_n}.$ By \cite[2.5(iii)]{JS}, $\Supp \frac{R}{I} = \Supp \frac{R}{\fa} \bigcup \Supp \frac{R}{\fb}.$ Let $1\leq i\leq n$ and assume that $\fp_i \in \Supp \frac{R}{\fa}.$ Then, $\Supp \frac{R}{\fp_i} \subseteq\Supp \frac{R}{\fa}$ and $\Ext^i_R(\frac{R}{\fp_i}, N)$ is finitely generated. Therefore, $\Ext^i_R(M, N)$ is finitely generated for all $i\geq 0$ and the result follows from the fact that $\Supp \frac{R}{I}= \Supp M$.

   \end{proof}
The following corollary presents an equivalent condition for the $\fa$-cofiniteness of $H^i_{\fa}(R)$ in the case where $\fa$ is a linked ideal.
\begin{cor} \label{c2}
Let $i\in \N_0$ and $I$ be a non-prime ideal of $R$ such that $\fa$
is linked by $I$. If $H^i_{\fa}(R)$ is $I$-cofinite then
$H^i_{\fa}(R)$ is $\fa$-cofinite. In particular, in the case where
$i> \grad I$ and $\fa\sim_{(I;R)}\fb$, $H^i_{\fa}(R)$ is
$\fa$-cofinite and $\fb$-cofinite if and only if it is $I$-cofinite.
\end{cor}

\begin{proof}

     If $i> \grad I$ then it is straight-forward to see that $H^i_{\fa}(R)$ is $\fb$-torsion and $\Supp H^i_{\fa}(R) \subseteq V(\fb)$. Now, the result follows from the above theorem.

   \end{proof}
\section{ attached prime ideals }


Let $(R,\fm)$ be a local ring. For all $i \in \mathbb{N}$, the family $\{ H^i_{\fm}(\frac{M}{\fa^nM})\}_{n \in\mathbb{N}}$ forms an inverse system. The inverse limit $F^i_{\fa} (M): = \underleftarrow{\lim}_n H^i_{\fm}(\frac{M}{\fa^nM})$ is called the $i$-th formal local cohomology module of $M$ with respect to $\fa$. Formal local cohomology were used by Peskine and Szepiro in \cite{PS1} in order to solve a conjecture of Hartshorne.

Artinianness of local cohomology and formal cohomology modules is one of the main problems in this subject, see for example \cite{BR}, \cite{DY1} and \cite{Sc2}. In this section we consider this problem.

The following modification of \cite[Theorem A]{DY1} will be used several times in the paper. We bring it here for the reader's convenience.
\begin{lem}\label{l2}
Let $R$ be a complete local ring. Then $$\Att H^{\dim M}_{\fa} (M) = \{\fp \in \Assh M | \sqrt{\fa+\fp} = \fm \}.$$
\end{lem}
\begin{proof}
In view of \cite[Theorem A]{DY1}, let $\fp \in \Att H^{\dim M}_{\fa} (M) = \{\fp \in \Ass M | \cd (\fa, \frac{R}{\fp})=\dim M \}.$ Then, by the Grothendieck's vanishing theorem \cite[6.1.2]{BS}, $\fp\in \Assh M.$ Now, the result follows from \cite[8.2.3]{BS}.

\end{proof}

In the following lemma, which will be used in the next theorem, we study the associated and attached prime ideals of local cohomology and formal local cohomology modules at the ''lowest'' and ''highest'' level.

\begin{lem}\label{l08}
Let $(R,\fm)$ be a complete local ring and $M$ be a finitely generated R-module of dimension $n$ . Then the following statements hold.
 \begin{itemize}
   \item [(i)] $ \Att H^ n _{\fa\cap \fb}(M) = \Att H^ n _{\fa}(M) \cap \Att H^ n _{\fb}(M) $.
   \item [(ii)] $ \Att F^ n _{\fa\cap \fb}(M) = \Att F^ n _{\fa}(M) \cup \Att F^ n _{\fb}(M) $.
   \item [(iii)] $ \Att F^ n _{\fa+ \fb}(M) = \Att F^ n _{\fa}(M) \cap \Att F^ n _{\fb}(M) $.
   \item [(iv)]  $\Ass F^0_{\fa\cap \fb} (M) = \Ass F^ 0 _{\fa}(M) \cap \Ass F^ 0 _{\fb}(M)$.

Moreover, if $\fa\fb M = 0$ then,

   \item [(v)] $ \Att H^ n _{\fa+ \fb}(M) = \Att H^ n _{\fa}(M) \cup \Att H^ n _{\fb}(M) $.
   \item [(vi)]  $\Ass F^0_{\fa+ \fb} (M) = \Ass F^ 0 _{\fa}(M) \cup \Ass F^ 0 _{\fb}(M)$.
   \end{itemize}
\end{lem}

\begin{proof}
\begin{itemize}
   \item [(i)] Let $\fp\in \Assh M$. Then, it is straight-forward to see that $\sqrt{(\fa \cap\fb) + \fp} = \fm$ if and only if $\sqrt{\fa +\fp} = \fm$ and $\sqrt{\fb +\fp} = \fm$. Now, the result follows from \ref{l2}.

   \item [(ii)] and (iii) are immediate by \cite[3.1]{BR}.
   \item [(iv)] As we have seen in the proof of the part (i), for any $\fp\in \Ass M$, $\sqrt{(\fa \cap\fb) + \fp} = \fm$ if and only if $\sqrt{\fa +\fp} = \fm$ and $\sqrt{\fb +\fp} = \fm$. Now, the result follows from \cite[4.1]{Sc2}.

   \item [(v)] Let $\fp \in \Att H^ n _{\fa}(M)$. Then, by the above lemma, $\fp\in \Assh M $ and $\sqrt{\fa +\fp} = \fm$. Hence, $\sqrt{\fa+\fb +\fp} = \fm$ and so, $ \fp \in \Att H^ n _{\fa+\fb}(M)$.
       Conversely, let $\fp\in \Att H^ n _{\fa+ \fb}(M)$. Then $\fp\in \Assh M $ and $\sqrt{\fa +\fb +\fp} = \fm$. Hence, by the assumption, $\fp \supseteq \fa$ or $\fp \supseteq \fb$ and so, $\sqrt{\fb +\fp} = \fm$ or $\sqrt{\fa +\fp} = \fm$. This implies that $\fp \in \Att H^ n _{\fa}(M) \cup \Att H^ n _{\fb}(M).$
   \item [(vi)]  The result follows from \cite[4.1]{Sc2}.
   \end{itemize}

\end{proof}

 The following theorem gives us a necessary and sufficient condition for $M$ to be Cohen-Macaulay in terms the existence of some special linked ideals over it.

\begin{thm} \label {t6}
Let $(R , \fm)$ be local and $d:= \dim M$. Then the following statements are equivalent.

 \begin{itemize}
   \item [( i )] $M$ is Cohen-Macaulay.
   \item [( ii )] There exist ideals $\fa, \fb$ and $I$ such that $\fa\sim_{(I;M)}\fb$ and $\Att H^d_{\fa}(M) \bigcap \Att H^d_{\fb}(M) \neq \varnothing$.
   \item [( iii )] There exist ideals $\fa, \fb$ and $I$ such that $\fa\sim_{(I;M)}\fb$, $\Ass F^0_{\fa}(\frac{M}{IM}) \bigcap \Ass F^0_{\fb}(\frac{M}{IM}) \neq \varnothing$ and $\mid \Ass \frac{M}{IM}\mid = 1$.
   \item [( iv )]There exist ideals $\fa, \fb$ and $I$ such that $\fa\sim_{(I;M)}\fb$, $\Ass \frac{M}{IM} = \Min\Ass \frac{M}{IM}$ and $\dim \frac{M}{\fa M} = 0$.
   \item [( v )] There exist ideals $\fa, \fb$ and $I$ such that $\fa\sim_{(I;M)}\fb$, $\Ass \frac{M}{IM} = \Min\Ass \frac{M}{IM}$ and $\Ass F^0_{\fa} (M) = \Ass M.$
   \end{itemize}
   In particular, $R$ is Cohen-Macaulay if and only if there exists a maximal $R$-sequence $\fx$ such that $\Ass \frac{R}{(\fx)}= \Min \frac{R}{(\fx)} $.
\end{thm}

\begin{proof}
 First of all, note that  if $M$ is Cohen-Macaulay then every system of parameters of $M$ is generated by an $M$-regular sequence. On the other hand, by \cite[2.2]{JS}, every $M$-regular sequence is $M$-selflinked.

$ (i)\rightarrow(ii) $ It is clear.

$ (ii)\rightarrow(i) $  By \ref{l08}(i), $\Att H^ d _{\fb\cap\fa}(M) \neq\varnothing$. Hence, by \cite[3.1]{JS}, $ht _M I = d$ and $M$ is Cohen-Macaulay.

$ (i)\rightarrow(iii) $ It is clear by \cite[4.1]{Sc2}.

$ (iii)\rightarrow(i) $ By \ref{l08}(iv), $\Ass F^ 0 _{\fb\cap\fa}(\frac{M}{IM}) \neq\varnothing$. Hence, by \cite[3.1]{JS}, $\Ass \Gamma_{\fm}(\frac{M}{IM}) \neq \varnothing$ and so $\fm \in \Ass (\frac{M}{IM})$. This implies that $\dim \frac{M}{IM}= 0.$ Therefore, $M$ is Cohen-Macaulay.

$ (iv)\rightarrow(i) $  The result follows from the concept of linkedness and \cite[2.9]{JS} .

$ (v)\leftrightarrow(iv) $ One has $\dim \frac{M}{\fa M} = 0$ if and inly if $\Ass F^0_{\fa} (M) = \Ass M.$ Now the result follows from \cite[4.1]{Sc2}.

For the end, we may assume that $\fm \neq (\fx)$. As $\fm \subseteq (\fx):_R (\fx):_R \fm$, $\fm$ is linked by $(\fx)$. Hence $\dim \frac{R}{(\fx)} = 0$, by \cite[2.9]{JS}, and $R$ is Cohen-Macaulay.

\end{proof}

The following lemma, which shows that the equidimensional property of $M$ passes through linkage to $\frac{M}{\fa M}$ and $\frac{M}{\fb M},$ will be used in the next proposition.

\begin{lem}\label{r1}
Assume that $M$ is equidimensional and $\fa$ and $\fb$ are geometrically linked over $M$ by zero ideal. Then $\frac{M}{\fa M}$ and $\frac{M}{\fb M}$ are equidimensional of dimension $\dim M.$
\end{lem}
\begin{proof}
Let $\fp \in \Min \Ass \frac{M}{\fa M}- \Min \Ass M$. Then, there exists $\fq \in \Min \Ass M$ such that $\fq \subset \fp$. Hence $\fa \nsubseteq \fq$ and $\fb \subseteq \fq$. This, in view of \cite[2.8(iv)]{JS}, implies that $\fp \in \Ass \frac{M}{\fa M} \cap V(\fb)= \Ass M \cap V(\fa +\fb)= \emptyset$, by \cite[2.8]{JS} and the fact that $\grad_M\fa +\fb>0.$ Therefore, $\Min \Ass \frac{M}{\fa M}\subseteq \Min \Ass M$ and the result follows.

\end{proof}
In the following proposition we study vanishing and attached prime ideals of $H^{\dim M}_{\fa} (M),$ where $\fa$ is linked over $M$.
\begin{prop}\label{l15}
Let $(R,\fm)$ be local and complete, $\fa\sim _{(0;M)}\fb$ and $n:= \dim M>0$. Then the following assertions hold.
\begin{itemize}
 \item [(i)] $\Att H^n_{\fa} (M) \subseteq \Assh \frac{M}{\fb M}$.

 \item [(ii)] If $\dim \frac{M}{\fa M} \neq\dim \frac{M}{\fb M}$ then either $H^n_{\fa} (M)=0$ or $H^n_{\fb} (M)=0$. In particular, if $\Ass F^0_{\fa} (M)= \Assh M$ then $H^n_{\fa} (M)\cong H^n_{\fm} (M)$ and $H^n_{\fb} (M)=0$.
 \item [(iii)] Assume that $M$ is equidimensional and $\fa$ and $\fb$ are geometrically linked over $M.$ Then $\Att H^n_{\fa} (M)= \Assh \frac{M}{\fb M}$ if and only if $\Att H^n_{\fb} (M)= \Assh \frac{M}{\fa M}$.
     \end{itemize}
\end{prop}

\begin{proof}

\begin{itemize}
 \item [(i)] Let $\fp \in \Att H^n_{\fa} (M)$. Then, by \ref{l2}, $\fp \in \Assh M$ and $\sqrt{\fa+\fp}=\fm$. Therefore, $\fa \nsubseteq\fp$ and, by \cite[2.8(i)]{JS}, $\fp \in \Assh \frac{M}{\fb M}$.

     \item [(ii)] Assume that $\dim \frac{M}{\fb M} <\dim \frac{M}{\fa M}$. Then, by \cite[2.5(iii)]{JS}, $\dim M =\dim \frac{M}{\fa M}$ and also, by (i), $\Att H^n_{\fa} (M) \subseteq \Assh \frac{M}{\fb M} \cap \Assh M = \emptyset$ which implies that $H^n_{\fa} (M)=0$.

     Now, let $\fp \in \Assh M$. Then, by the assumption and \cite[4.1]{Sc2}, $\sqrt{\fa+\fp}=\fm$ and $\fa \nsubseteq \fp$. Therefore, by (i), $\Att H^n_{\fb} (M)\subseteq \Assh M \cap \Assh \frac{M}{\fa M}= \emptyset$ and $H^n_{\fb} (M)=0$. On the other hand, in view of \ref{l2}, $\Att H^n_{\fa}(M) = \Att H^n_{\fm} (M)$ and, using \cite[1.6]{DY}, this implies that $H^n_{\fa}(M) \cong H^n_{\fm} (M)$.
 \item [(iii)] In view of \ref{r1}, $\frac{M}{\fa M}$ and $\frac{M}{\fb M}$ are equidimensional of dimension $\dim M$. Assume that $\Att H^n_{\fa} (M)= \Assh \frac{M}{\fb M}.$ Then, by \ref{l2}, for every $\fp \in \Min\Ass \frac{M}{\fb M}$, $\sqrt{\fa+\fp}=\fm$. Hence, for every $\fq \in \Min\Ass \frac{M}{\fa M}$, $\sqrt{\fq+\fp}=\fm$. This implies that $$\sqrt{\fq+\cap_{\fp \in \Min\Ass \frac{M}{\fb M}}\fp}=\sqrt{\fq+\fb}=\fm.$$ Therefore, by (i), $\Att H^n_{\fb} (M)= \Assh \frac{M}{\fa M}$.
     \end{itemize}

\end{proof}
\bibliographystyle{amsplain}

\end{document}